\documentclass[a4paper, 11pt]{article}
\usepackage{amsmath, amssymb,amscd, amsthm}

\usepackage[mathscr]{eucal}
\usepackage{fullpage}
\usepackage{url}
\usepackage{cancel}
\usepackage[margin=0.95in]{geometry}
\usepackage{ulem}

\newcommand\cyr{%
 \renewcommand\rmdefault{wncyr}%
 \renewcommand\sfdefault{wncyss}%
 \renewcommand\encodingdefault{OT2}%
\normalfont\selectfont} \DeclareTextFontCommand{\textcyr}{\cyr}

\newtheorem{theorem}{Theorem}
\newtheorem{lemma}[theorem]{Lemma}
\newtheorem{corollary}[theorem]{Corollary}

\newtheorem{question}[theorem]{Question}

\long\def\symbolfootnote[#1]#2{\begingroup%
\def\thefootnote{\fnsymbol{footnote}}\footnote[#1]{#2}\endgroup}

\title{Counting rectangles and an improved restriction estimate for the paraboloid in $F_p^3$}

\author{Mark Lewko}

\date{}

\begin{document}

\maketitle

\begin{abstract}Given $A \subset F_{p}^2$ a sufficiently small set in the plane over a prime residue field, we prove that there are at most $O_\epsilon (|A|^{\frac{99}{41}+\epsilon})$ rectangles with corners in $A$. The exponent $\frac{99}{41} = 2.413\ldots$ improves slightly on the exponent of $\frac{17}{7} = 2.428\ldots$ due to Rudnev and Shkredov. Using this estimate we prove that the extension operator for the three dimensional paraboloid in prime order  fields maps $L^2 \rightarrow L^{r}$ for $r >\frac{188}{53}=3.547\ldots$ improving the previous range of $r\geq \frac{32}{9}= 3.\overline{555}$. 
\end{abstract}

\symbolfootnote[0]{2010 Mathematics Subject Classification 42B10}

\section{Introduction}
We will use $F$ to denote a finite field of prime residues. We say that a triple of points $(x_0,x_1,x_2)$, each an element of the finite plane $F^{2} := F\times F$, is a corner if 
$$(x_{1}-x_{0})\cdot (x_{2}-x_{1})=0$$
and all three points do not lie on a single line\footnote{This condition is needed to prevent the points from lying on an isotropic line which will exist if and only if $-1$ is a square in $F$. See \cite{LewkoKakeya} for a discussion of isotropic lines in the context of related problems.}. In other words, this relation states that the vectors $v_1-v_0$ and $v_2-v_1$ are orthogonal or ``form a right angle'' at $x_1$. We say that a quadruple of points $R=(r_1,r_2,r_3,r_4)$ is a rectangle if $(r_i,r_{i+1},r_{i+2})$ is a corner for $i\in\{0,1,2,3\}$ with the index arithmetic modulo $4$. Again, intuitively, this condition is an analog of each corner of the rectangle being at a right angle, and coincides with the usual definition of a rectangle in the Euclidean setting. Let $A \subseteq F^2$ be a point set in the finite plane over $F$. We consider two questions:
\begin{question}\label{q:Corners}How many triples of $A$ form a corner?
\end{question}
\begin{question}\label{q:Rects}How many quadruples of $A$ form a rectangle?
\end{question}
The answer to the second question is at most $4$ times the answer to the first. Without any restrictions on $A$, the optimal answer is $|A|^{\frac{5}{2}}$ in both cases. The optimality is easily seen by considering $A=F^2$. However, when $A$ is sufficiently small compared to $p^2$ one expects to be able to improve this and indeed it seems reasonable to conjecture $O_{\epsilon}( |A|^{2+\epsilon})$ for sufficiently small sets in the case of both questions. In the Euclidean plane this was proved by Pach and Sharir  \cite{PS} using the Szemer\'edi-Trotter theorem in 1992. In the finite field case, the first problem was recently considered by Rudnev and Shkredov who obtained the following estimate:
\begin{theorem}\label{thm:RScorners}Let $F$ be a prime order field is not a square and $A \subset F^2$ such that $|A|\leq p^{\frac{26}{21}}$. Then the number of corners (or rectangles) in $A$ is $O(|A|^\frac{17}{7})$.
\end{theorem}
As is usual with finite field results of this type, in most applications the set $A$ is safely smaller than the hypothesis $|A|\leq p^{\frac{26}{21}}$ and this exponent could probably be optimized further. Rudnev and Shkredov's argument closely follows Pach and Sharir's which reduces the problem to estimates for the number of $k-$rich lines that can intersect a point set. In the Euclidean case optimal estimates follow from the Szemer\'edi-Trotter theorem. In the finite field setting, an optimal form of the Szemer\'edi-Trotter theorem is not known. Rudnev and Shrkedov's result, however, is what is obtained when one runs the Pach-Sharir method using the best available Szemer\'edi-Trotter theorem available in the setting, due to Stevens and de Zeeuw \cite{Sd}. In some sense the Rudnev-Shrkedov adaptation of the Pach-Sharir machine is optimal, as an optimal Szemer\'edi-Trotter theorem in this setting would, as in the Euclidean case, gives an optimal bound of $O_\epsilon (|A|^{2+\epsilon})$ for question 1. Our first result is the following:
\begin{theorem}\label{thm:mainRect}Let $F$ be a prime order field and $A \subset F^2$ such that $|A|\leq p^{\frac{26}{21}}$. Then the number of rectangles in $A$ is $O_{\epsilon}(|A|^{\frac{99}{41}+\epsilon})$.
\end{theorem}
This improves on Rudnev and Shkredov's result but only for the more constrained problem of counting rectangles instead of corners. In fact exploiting the additional constraint that arises from considering all four vertices is the main novelty of our argument. Indeed we make no progress on the general finite field Szemer\'edi-Trotter incidence problem. Roughly speaking we proceed by showing, using the fourth corner of the rectangles being counted, that a hypothetical set for which the rectangle estimate implied by Theorem \ref{thm:RScorners} is sharp must concentrate on a grid. We are then able to apply a stronger incidence estimate for Cartesian product sets, again due to Stevens and de Zeeuw, to obtain a slight improvement in that case. This estimate is one of the many recently discovered consequences of Rudnev's point-plane incidence bound \cite{Rudnev}.

We now turn to our main application. The restriction problem is a central open problem in Euclidean harmonic analysis. Finite field variants of this problem were posed by Mockenahupt and Tao \cite{MT} in 2002 and has been the subject of a large number of recent papers. Many of these papers give a detailed overview and discussions of the problem and survey the literature, so we will not repeat this material here. See \cite{MT}, \cite{IKL}, \cite{LewkoEnd}, and \cite{LewkoKakeya}, and sections \ref{sec:Not} and \ref{sec:Res} below for notation. In this setting we obtain the following improved restriction estimate for the paraboloid in $F^3$.
\begin{theorem}\label{thm:NewRest}Let $F$ be a prime order field in which $-1$ is not a square and let $P:=\{(\underline{x},\underline{x}\cdot \underline{x}): \underline{x} \in F^{2}\}$ denote the paraboloid in $F^3$. Define the Fourier extension operator associated to $P$, mapping functions on $P$ to functions on $F^3$, by $ (fd\sigma)^{\vee}(x) := \frac{1}{|P|} \sum_{\xi \in P} f(\xi)e\left(x \cdot \xi \right)$. Then one has the inequality
$$ \|(fd\sigma)^\vee\|_{L^r(F^3)} \lesssim \|f\|_{L^2(P,\,d\sigma)}$$
for $r > \frac{188}{53}$. 
\end{theorem} This improves the prior estimate of $p\geq \frac{32}{9}$ due to Rudnev and Shkredov \cite{RS}. As remarked there, the conjectured bound $O_{\epsilon} (n^{2+\epsilon})$ on question 2 (for sufficiently sized sets) would imply the range $p>\frac{10}{3}$ in the above theorem, which still falls short of the conjectured range of $p\geq3$. 

Theorem \ref{thm:mainRect} also has some application to the analysis of two-source extractors with exponentially small error, which is of interest in theoretical computer science. We refer the reader to \cite{LewkoExt}, \cite{B}, and \cite{CZ} for definitions and background discussion. Here the main interest is in constructing explicit two-source extractors with a provable min-entropy rate as small as possible and exponentially small error\footnote{In a recent breakthrough paper, Chattopadhyay and Zuckerman \cite{CZ} have constructed two-source extractors with arbitrarily small min-entropy and poly-logarithmically small error.}. Inserting our result into Proposition 7 of \cite{LewkoExt} implies that Bourgain's 3-d paraboloid extractor extracts from sources with min-entropy greater than $\frac{123}{260}=.473\ldots$. This improves the analysis given there which produced a rate near $\frac{21}{41}=.477\ldots$. Currently the best-known (provable) construction extracts from sources with min-entropy rate $> \frac{4}{9}$ and is obtained by the analog of Bourgain's construction with the $3$-dimensional discrete paraboloid replaced by the the $4$-dimensional discrete paraboloid. While that extractor might well continue to work for lower entropy sources, the min-entropy rate of $\frac{4}{9}$ is the limitation of the method, at least in its present form. On the other hand, as discussed there, an exponent of $2+\epsilon$ in the main theorem here would give a min-entropy rate near $\frac{3}{8}=.375$ for Bourgain's original $3$-dimensional extractor. Any exponent less than $\frac{37}{16}=2.3125$ in Theorem \ref{thm:mainRect} would achieve a min-entropy rate lower than $\frac{4}{9}$.

Lastly we note that Theorem \ref{thm:RScorners} was recently used by Iosevich, Koh, Pham, Shen, and Vinh \cite{IKPSV} to obtain an improved exponent on the Erd\"os distance problem in finite fields. Our estimate can be incorporated into their argument to obtain a slightly better exponent for that problem. However, that result would fall short of the even more recent paper of Lund and Petridis \cite{LP} which proceeds somewhat differently. Larger improvements to Theorem \ref{thm:RScorners} would certainly led to further improvements however. 

\textbf{Note added:} After this note appeared as a preprint, Iosevich, Koh and Pham \cite{IKP} used Theorem \ref{thm:mainRect} with some new geometrical ideas to obtain an improved exponent on the distance problem over the Lund and Petridis result. This has been yet further superseded by an even more recent argument of Murphy, Rudnev and Stevens \cite{MRS}.

\section{Notation}\label{sec:Not}
We will write $X \sim Y$ to indicate that $Y/2 \leq X \leq 2Y$. For example $D_{\lambda } := \{x \in D : f(x) \sim \lambda\}$ would denote the elements of the domain, $D$, of $f$ where $\lambda/2 \ \leq f(x) \leq 2 \lambda $. We will also use the notation $a \lesssim b$ to indicate that the inequality $a \leq c b$ holds with some universal constant $c$. Let $F^d$ denote the $d$ dimensional vector space over $F$. We write $L^r(F^3)$ to denote the $L^r$ norm with respect to the counting measure on $F^3$ and $L^p(P,\,d\sigma)$ to denote the $L^p$ norm on the set/surface $P$ with respect to the normalized counting measure on $P$ which assigns a mass of $|P|^{-1}$ to each point in $P$. Furthermore we let $e(\cdot): F \rightarrow C$ denote a nontrivial additive character on $F$.

\section{Incidence estimates}
Non-trivial incidence estimates in the finite field setting were first obtained by Bourgain, Katz and Tao in 2005, as a corollary to their sum-product estimate. See \cite{BKT} and \cite{TV}. We will make use of the much stronger recent Szemer\'edi-Trotter-type incidence results of Stevens and de Zeeuw \cite{Sd} which, in turn, makes use of work of Rudnev \cite{Rudnev} and Koll\'{a}r \cite{K}.

\begin{theorem}\label{thm:sd}Let $A \subseteq F_{p}\times F_{p}$, $L$ a set of lines in $F_{p}\times F_{p}$ and $I(A,L)$ the set of incidences between points in $A$ and lines in $L$. Then for $|A|^{13}|L|^{-2} \leq p^{15}$ we have
$$|I(A,L)| \lesssim |A|^{\frac{11}{15}}|L|^{\frac{11}{15}} + |A| + |L|.$$
\end{theorem}Given $A$ and $L$ we will denote by $L_k$ the set of lines that have multiplicity $k$. Combining the estimate above with the Vinh's \cite{V} estimate $|I(A,L)| \leq \frac{|A||P|}{p} + (|A||L|p)^{1/2}$ for unrestricted $|A|$, one has that the number of $k$-rich lines in set $A \subseteq F_{p}\times F_{p}$, say $L_k$, satisfies
\begin{equation}\label{eq:Lkstrong}
|L_k| \lesssim n^{\frac{11}{4}} k^{-\frac{15}{4}} + n k^{-1} + n^{\frac{13}{2}} p^{\frac{15}{2}}.
\end{equation}
Following \cite{RS} we will use the following cruder estimate which simplifies the presentation without introducing any inefficiency to the final result
\begin{equation}\label{eq:LkCrude}
|L_k| \lesssim n^{\frac{11}{4}} k^{-\frac{15}{4}} + n^{\frac{5}{4}} k^{-1}.
\end{equation}
We will also need the stronger estimate of Stevens and de Zeeuw which applies when $A$ is a Cartesian product:
\begin{theorem}\label{thm:sdProd}Let $A,B \subseteq F_{p}$, $L$ a set of lines in $F_{p}^2$, and $I(A \times B,L)$ the set of incidences between points in $A \times B$ and lines in $L$. Then for $|A||L| \leq p^{2}$ we have
$$|I(A,L)| \lesssim |A|^{\frac{3}{4}} |B|^{\frac{1}{2}} |L|^{\frac{3}{4}} + |L|.$$
\end{theorem}We will use this result in the following slightly more general form: 
\begin{corollary}Let $\ell_A$ and $\ell_B$ be distinct non-parallel lines in $F^2$. Let $A \subset \ell_A$, $B \subset \ell_B$, and $S = \{a+b : a \in A, b \in B \}$. Then:
$$|I(S,L)| \lesssim |A|^{\frac{3}{4}} |B|^{\frac{1}{2}} |L|^{\frac{3}{4}} + |L|.$$
\end{corollary}
\begin{proof}Clearly incidence counts are preserves under linear transformations which map points to points, lines to lines. By translation, we may assume that that is $\ell_A$ and $\ell_B$ intersect at the origin. It then  follows that there is a linear transformation that takes $\ell_A$ to the $x$-axis and $\ell_B$ to the $y$-axis. Thus the corollary follows from theorem \ref{thm:sdProd}.
\end{proof}

\begin{lemma}\label{lem:cells} Let $B \subseteq F^2$ that is contained within the union of $m$ $k$-by-$k$ grids. Then the number of $j$-rich lines, $L_j$, is at most $\lesssim k^5 (j^{-4}m^{4}+1)$.
\end{lemma}
\begin{proof}The $j$-rich lines $L_j$ must create $\sim j|L_j|$ with $B$. Thus the lines $L_j$ must make $km^{-1}|L_j|$ incidences with one of the $k$-by-$k$ grids. Applying Theorem \ref{thm:sdProd} we have
$jm^{-1}|L_j| \leq k^{\frac{5}{4}}|L_j|^{\frac{3}{4}}+|L|.$ Rearranging terms gives the claim.
\end{proof}

\section{Counting rectangles}
Given $A \subseteq F^2$ we will let $\square(A)$ denote the number of rectangles with vertices in $A$. We first recall that if $A_k \subseteq F^2$ is a collection of sets and $A = \bigcup A_k$ then 
\begin{equation}\label{eq:Qtri}
( \square(A) )^{1/4} \leq \sum_{k} ( \square(A_k) )^{1/4}. 
\end{equation}This can be proved combinatorially, however it is perhaps more intuitively follows from the fact that the quantity is a multiple of a fourth power of the $L^4$ norm. See \eqref{eq:AEr} and the subsequent discussion below. This allows us to split the set $A$ into $|A|^{\epsilon}$ subsets and prove the result for each individual set. We decompose $A$ as a disjoint union $A = A' \cup  \bigcup_{k} A_k$ for $1\leq k \leq \frac{1}{41} \log n$ as follows. We construct $A_1, A_2, \ldots, A_{\frac{1}{41} \log n}$ sequentially by removing selected points from $A$ using the greedy selection algorithm so that:
\begin{enumerate}
\item For $1\leq i \leq \frac{1}{41} \log n$, each $A_i$ will be the union of the intersection of at most $O(n^{\frac{7}{41}}2^{i})$  $n^{\frac{17}{41}}$-by-$n^{\frac{17}{41}}$ grids, $G$, that intersects $A$ so that $|A \cap G| \sim 2^{-i} n^{\frac{34}{41}}$. 
\item For any $n^{\frac{17}{41}}$-by-$n^{\frac{17}{41}}$ grid $G$  we have that $| A'  \cap G| \lesssim n^{-\frac{1}{41}} |G|$.
\end{enumerate}
It further follows from the pigeonhole principle that if $G$ is a $j$-by-$j$ grid with $j \geq n^\frac{17}{41}$ then 
\begin{equation}\label{eq:GridDen}
|A' \cap G| \lesssim n^{-\frac{1}{41}}|G|.
\end{equation}For a rectangle $R$ with vertices in $A$, associate to $R$ a line $\ell_{R}$ such that $\ell_{R}$ coincides with an edge of $R$ and $|\ell_{R}\cap A|$ is maximized over the four such choices of lines. Thus we can dyadically decompose

$$\square(A) = \sum_{\substack{k \text{ dyadic} \\ k \leq n^2} } |\{R \in A : |\ell_{R}| \sim k \}|.$$
We split this sum as follows:
$$ \square(A)= \sum_{\substack{k \text{ dyadic} \\ k \leq n^{\frac{17}{41}}}}  |\{R \in A : |\ell_{R}| \sim k \}| +  \sum_{\substack{k \text{ dyadic} \\ n^{\frac{17}{41}} \leq k \leq  n^{\frac{6}{11}}   }}  |\{R \in A : |\ell_{R}| \sim k \}| $$
$$+  \sum_{\substack{k \text{ dyadic} \\ k \geq n^{\frac{6}{11}}   }}  |\{R \in A : |\ell_{R}| \sim k \}|  =:  I + II + III.$$
The only estimate in the above decomposition where we need to proceed based on the decomposition just described is in step $II$.  The analysis of terms I and III will be carried out on all of $A$ and apply to each set in the decomposition of $A$. Given $A$ and a point $x \in A$ let $\tilde{L_{x}}$ denote the set of lines through $x$. Given $\ell \in \tilde{L_{x}}$ let $\ell_{x}^{\bot}$ the perpendicular\footnote{By hypothesis we omit the isotropic case when $\ell = \ell_x^\bot$.} line to $\ell$ that passes through $x$. Let $L_{x}$ denote the set of lines such that $|\ell \cap A| \geq |\ell_{x}^{\bot} \cap A|$, with an arbitrary choice made if the two lines contain the same number of points. Given a line $\ell$ we let $n(\ell):= |A \cap \ell|$. Now, proceeding as in \cite{RS}, we can estimate 
$$ I \lesssim \sum_{x \in A} \sum_{\substack{\ell \in L_{x} \\ n(\ell) \leq n^{\frac{17}{41}}}} n(\ell) n(\ell_{x}^{\bot}) \lesssim |A|^2 n^{\frac{17}{41}} \lesssim n^{\frac{99}{41}}.$$
By \eqref{eq:LkCrude}, the number of $k$-rich lines in $A$ is at most $n^{5/4} k^{-1}$ once $k \geq n^{\frac{6}{11}}$. We can crudely estimate III as
$$III \lesssim \sum_{\substack{ k \text{ dyadic} \\ k \geq n^{\frac{6}{11}} }}\sum_{x \in A} \sum_{\substack{ \ell \in L_{x} \\ n(\ell) \sim k }}  n(\ell) n(\ell_{x}^{\bot}) \leq \sum_{\substack{ k \text{ dyadic} \\ k \geq n^{\frac{6}{11}} }} n \times  n^{\frac{5}{4}} k^{-1} \times k \lesssim n^{\frac{9}{4}} \log n.  $$

We are left to consider the contribution from II. This is where our argument exploits that all four vertices of each rectangle must be included in $A$, and the decomposition described earlier. Given $x \in A$, $\ell \in x$ and $\ell^{\bot}_{x}$ one can bound the number of rectangles with $x$ as a vertex and with sides adjacent to $x$ contained in $\ell$ and $\ell^{\bot}_{x}$, say $\Delta(\ell,x)$, as $\Delta(\ell,x) \leq n(\ell) n (\ell_{x}^{\bot})$, and this is implicitly how the Rudnev-Shkredov argument proceeds. We observe, however, that $\Delta(\ell,x)$ is in fact equal to the size of the intersection of a $k$-by-$k$ grid with $A'$. By \eqref{eq:GridDen} we have
$$\Delta(x,\ell) \lesssim  n^{-\frac{1}{41}}k^2 . $$
Using this and the estimate $|L_{k}|\lesssim n^{\frac{11}{4}} k^{-\frac{15}{4}}$ from \eqref{eq:LkCrude}, we have
$$\square(A_0) \lesssim  I + III + \sum_{\substack{k \text{ dyadic} \\ n^{\frac{17}{41}} \leq k \leq  n^{\frac{6}{11}}}} \sum_{x \in A_{0}} \sum_{\ell \in L_{k}} \Delta(\ell,x) 
\leq  \sum_{\substack{k \text{ dyadic} \\ n^{\frac{17}{41}} \leq k \leq  n^{\frac{6}{11}}}} n^{-\frac{1}{41}} k^3 |L_k| $$
$$\leq  \sum_{\substack{k \text{ dyadic} \\ n^{\frac{17}{41}} \leq k \leq  n^{\frac{6}{11}}}}  n^{-\frac{1}{41}} k^3 \times n^{\frac{11}{4}} k^{-\frac{15}{4}} 
\lesssim n^{99/41}.$$
Lastly we consider the II contribution for each $A_i$. To simplify notation, let $\lambda = 2^{-i}$. Thus $1\geq \lambda \geq n^{-\frac{1}{41}}$. Since $A_i$ is contained in the union of at most $O(n^{\frac{7}{41}}\lambda^{-1})$ $n^{\frac{17}{41}}$-by-$n^{\frac{17}{41}}$ grids, Lemma \ref{lem:cells} gives that the number of $k$-rich lines intersecting $A_i$ is $k^{-4}n^\frac{113}{41} \lambda^{-4}$. Proceeding as above we have

$$\square(A_i) \lesssim I + II + \sum_{\substack{k \text{ dyadic} \\ n^{\frac{17}{41}} \leq k \leq  n^{\frac{6}{11}}}} \lambda k^3 \times |L_k| \lesssim  n^{\frac{96}{41}}\lambda^{-3}\log n$$
since $\lambda \geq n^{-\frac{1}{41}}$ this completes the proof of the desired estimate for each $A_i$, and thus for $A$.

Putting everything together, using that $A = A' \cup  \bigcup_{k} A_k$ for $1\leq k \leq \frac{1}{41} \log n$ and the inequality \eqref{eq:Qtri}, we have that 
$$ \square(A)  \lesssim \square(A') + \left( \sum_{i=1}^{\frac{1}{41} \log n} (\square(A_i))^{1/4} \right)^{4} \lesssim \square(A') + \max_{1\leq i \leq \frac{1}{41} \log n} \square(A_i) \times \log^{O(1)} n.$$
On the other hand, the  analysis above shows that $\square (A')$, $\square (A_i) \lesssim n^{\frac{99}{41}} \log^{O(1)} n$. This completes the proof.

\section{A $3$-d restriction estimate}\label{sec:Res}
In this section we prove Theorem \ref{thm:NewRest}. We have nothing new to say about the Fourier analytic machinery, which has been exposited in detail in a number of related papers. See \cite{IKL}, \cite{MT}, or \cite{LewkoKakeya}, for instance. We therefor will keep our presentation very concise focusing on the numerology. Given an element $a=(\underline{a},\underline{a}\cdot\underline{a}) \subseteq P \subset F^3$ we will denote the projection onto $F^2$ as $\underline{a}$ and similarly if $A \subseteq P$ we denote the projection of the elements of $A$ onto $F^2$ as $\underline{A}$. 

Next we record the Mockenahupt-Tao machine which relates $\|\widehat{g}\|_{L^2(P, d\sigma)}$ to the additive energy of extension operator applied to certain level sets of $g$. This appears as stated in \cite{IKL} and in slightly different notation as Lemma 28 in \cite{LewkoKakeya} and is implicit in \cite{MT}. Given a function $g:F^3 \to \mathbb C$ we will denote its support by $G$. We will also use $G$ to refer to the characteristic function of $G$. Moreover we define $\tilde{G_z} : F^{2} \rightarrow \{0,1\}$ by $\tilde{G_z}(\underline{x})= G(\underline{x},z)$. Finally we define $G_z: P \to \{0,1\}$ by $G_z(\underline{x},\underline{x}\cdot \underline{x})=\tilde{G_z}(\underline{x})$ (and $0$ for $x \notin P$). With this notation we have the following

\begin{lemma}\label{lem:MT0} Let $g : F^3 \rightarrow C$ such that $|g| \lesssim 1$ on its support. We then have
$$ \|\widehat{g}\|_{L^2(P, d\sigma)} \lesssim |supp(g)|^{\frac{1}{2}} + |supp(g)|^{\frac{3}{8}} |F|^{\frac{1}{2}}\left( \sum_{z \in F}||(G_z d\sigma)^\vee ||_{L^4(F^3)} \right)^{\frac{1}{2}} .$$
\end{lemma}

Next we note that the $L^4$ norm of the extension operator is a multiple of the additive energy of the set. Recall that for $S\subseteq F^3$ the additive energy of $S$ is defined to be $\Lambda(S) := \sum_{\substack{ a,b,c,d \in S \\ a+b=c+d}} 1$. Then we have for a function $g: P \rightarrow C$ such that $|g| \sim 1$ on its support $G$ that
\begin{equation}\label{eq:AEr}
||(gd\sigma)^{\vee}||_{L^{4}(F^3)} \sim |F|^{-\frac{5}{4}} (\Lambda(G))^{1/4} 
\end{equation}
where we have identified the set $G$ with its characteristic function. Now the next observation is that for $a,b,c,d \in P$ the relation $a+b=c+d$ implies $a+b-c\in P$ and this holds if and only if
$(\underline{a}-\underline{c}+\underline{b})\cdot (\underline{a}-\underline{c}+\underline{b}) = \underline{a}\cdot\underline{a} -\underline{c}\cdot\underline{c} +\underline{b}\cdot\underline{b}$, which holds if and only if
$$(\underline{a}-\underline{c})\cdot (\underline{c}-\underline{b})=0.$$ 
In other words $(a,b,c)$ must satisfy the algebraic relation defining a corner, albeit without a guarantee that the points do not lie on a line. Cycling through the analogous relations for other triples of $\{a,b,c,d\}$, for all $S\subseteq P$, it follows that $\Lambda(S) = \square(S)+\square'(S)$, where $\square'(S)$ denotes the number of ``rectangles" / quadruples of points satisfying the above algebraic relation, but all of which lie on a line\footnote{To the best of our knowledge the precise connection with rectangle counting first appears in Bourgain and Demeter's paper \cite{BD} in the context of the restriction problem for discrete/lattice paraboloid.}. If $-1$ is not a square in $F$, then for $x \in F^2$ we have $x \cdot x=0$ if and only if $x=0$. Thus in this case $\square'(S) \lesssim n$. Without this hypothesis $\square'(S)$ can be as large as $|S|^{3}$ if $S$ are the points on an isotropic line.  

Now by a dyadic decomposition and the $\epsilon$-removal lemma (see Lemma 1.2 in \cite{IKL}) it suffices to prove the result for functions $g:F^3 \rightarrow C$ that are $|g| \sim 1$ on its support $G$.  We collect various estimates which control $\|\widehat{g}\|_{L^2(P, d\sigma)}$ in terms of $G$. We start with the Stein-Tomas-type estimate (see 1.7 in \cite{IKL}):
\begin{equation}\label{eq:STs}
\|\widehat{g}\|_{L^2(P, d\sigma)} \lesssim |G|^{\frac{1}{2}}  + |G| |F|^{-\frac{1}{2}}.
\end{equation}
If $|G|\leq |F|^{\frac{94}{53}}$ this implies $\|\widehat{g}\|_{L^2(P, d\sigma)} \lesssim |G|^{\frac{135}{188}}$ Next we observe that the relation \eqref{eq:AEr} combined with our main result gives the estimate $||(gd\sigma)^{\vee}||_{L^{4}(F^3)} \lesssim_{\epsilon} |F|^{\epsilon} |F|^{-\frac{5}{4}} |G|^{\frac{99}{164}} $ for $|G|\leq |F|^{\frac{26}{21}}$. Combining this with Lemma \ref{lem:MT0} implies, assuming $G$ satisfies $|G_z| \leq |F|^{\frac{26}{21}}$ for each $z \in F$, that
$$  \|\widehat{g}\|_{L^2(P, d\sigma)} \lesssim_{\epsilon} |G|^{1/2} +   |G|^{3/8} |F|^{\frac{1}{2} -\frac{5}{8}+\epsilon} \left( \sum_{z \in F} |G_z|^{\frac{99}{164}} \right)^{\frac{1}{2}} $$ 
\begin{equation}\label{eq:MTn}
\lesssim_{\epsilon} |G|^{1/2} +   |F|^{\frac{3}{41}+\epsilon} |G|^{\frac{111}{164}}.
\end{equation}
This also gives $\| \widehat{g}\|_{L^2(P, d\sigma)} \lesssim |G|^{\frac{135}{188}}$, provided $|G| \geq |F|^{\frac{94}{53}}$ and for each slice we have $|G_z| \leq |F|^{\frac{26}{21}}$. Next, using the relation \eqref{eq:AEr} with the ``Cauchy-Schwarz incidence estimate" $\Lambda(A)\lesssim |A|^{\frac{5}{2}}$ valid for all $A$ (see \cite{MT} or \cite{LewkoImp}), we have that 
$$  \|\widehat{g}\|_{L^2(P, d\sigma)} \lesssim |G|^{1/2} +   |G|^{3/8} |F|^{-\frac{1}{8}} \left( \sum_{z \in F} |G_z|^{\frac{5}{8}} \right)^{\frac{1}{2}}.$$ 
Applying H\"older's inequality to the inner sum we have
$$  \|\widehat{g}\|_{L^2(P, d\sigma)} \lesssim |G|^{1/2} +   |G|^{11/16} |F|^{\frac{1}{16}}.$$ 
For $|G|\geq |F|^{\frac{47}{21}} \geq |F|^{\frac{517}{493}}$ this implies $\|\widehat{g}\|_{L^2(P, d\sigma)} \lesssim ||G|^{\frac{135}{188}}$. 
Finally we consider $|F|^{\frac{26}{21}} \leq  |G| \leq |F|^{\frac{47}{21}}$ such that $|G_z| \geq |F|^{\frac{26}{21}}$ or $0$ for each $z$. Clearly the number of non-empty slices is $|G| |F|^{-\frac{26}{21}} \leq |F|$. Now repeating the above process, but using H\"older with the set $Z$, we have
$$  \|\widehat{g}\|_{L^2(P, d\sigma)} \lesssim |G|^{1/2} +   |G|^{3/8} |F|^{-\frac{1}{8}} \left( \sum_{z \in Z} |G_z|^{\frac{5}{8}} \right)^{\frac{1}{2}}$$ 
$$  \lesssim |G|^{1/2} +    |G|^{3/8} |F|^{-\frac{1}{8}} |G|^{\frac{5}{16}} \left(|G| |F|^{-\frac{26}{21}} \right)^{\frac{3}{16}}$$ 
$$  \lesssim |G|^{1/2} +    |G|^{7/8} |F|^{-\frac{5}{14}} \lesssim |G|^{\frac{269}{376}} \lesssim |G|^{\frac{135}{188}}.$$
Where, in the last inequality, we have used the condition $|G|\leq |F|^{\frac{47}{21}}$. Collecting our results we have proven that for all level set functions $g$ we have
$$\|\widehat{g}\|_{L^2(P, d\sigma)}  \lesssim_{\epsilon} |G|^{\frac{135}{188}+\epsilon}.$$
The main result now follows from duality and $\epsilon$-removal, as previously stated.

\texttt{M. Lewko}

\textit{mlewko@gmail.com}

\end{document}